\documentclass[12pt]{amsart}
\usepackage{graphicx}
\usepackage{amsmath, amsxtra, amssymb, latexsym}
\usepackage{hyperref}
\newtheorem{thm}{Theorem}[section]

\newtheorem{lem}[thm]{Lemma}

\theoremstyle{definition}
\newtheorem{defn}[thm]{Definition}
\theoremstyle{remark}

\numberwithin{equation}{section}

\def\proof{\textbf{Proof}.\,\,}
\def\endproof{\hfill$\Box$\\}

\def\XXint#1#2#3{{\setbox0=\hbox{$#1{#2#3}{\int}$}
  \vcenter{\hbox{$#2#3$}}\kern-.5\wd0}}

\begin{document}

\title[Basic cohomology group decomposition]
      {Basic cohomology group decomposition of K-contact 5-manifolds}

\author[J. Zhou; P. Zhu]{Jiuru Zhou; Peng Zhu}
\address[J. Zhou]{School of Mathematical Sciences, Yangzhou
University, Yangzhou, Jiangsu, China 225002.}
\email{\href{mailto: J.
Zhou<zhoujr1982@hotmail.com>}{zhoujr1982@hotmail.com}}

\address[P. Zhu]{School of Mathematics and Physics,
Jiangsu University of Technology, Changzhou,
Jiangsu, People¡¯s Republic of China 213001.}
\email{\href{mailto: P. Zhu<zhupeng2004@126.com>}{zhupeng2004@126.com}}



\begin{abstract}
In this paper, we consider decompositions of basic degree 2 cohomology for a compact K-contact 5-manifold $(M,\xi,\eta,\Phi,g)$, and conclude the pureness and fullness of $\Phi$-invariant and $\Phi$-anti-invariant cohomology groups. Moreover, we discuss the decomposition of the complexified basic degree 2 cohomology group. This is an analogue problem when Draghici, Li and Zhang \cite{DLZ1} considered the $C^{\infty}$ pureness and fullness of $J$-invariant and $J$-anti-invariant subgroups of the degree 2 real cohomology group $H^2(M,\mathbb{R})$ of any compact almost complex manifold $(M, J)$.
\end{abstract}

\maketitle

\section{Introduction}
Donaldson \cite{Don} posed a question: for an almost complex structure $J$ on a compact 4-manifold $M$ which is tamed by a symplectic form $\omega$, is there a symplectic form compatible with $J$? In order to study this question, Li and Zhang \cite{LZ}, Draghici, Li and Zhang \cite{DLZ1,DLZ2} investigated the decomposition of the real degree two de Rham cohomology group $H^2(M,\mathbb{R})$, and introduced $J-$invariant and $J-$anti-invariant subgroups $H_J^{+}(M)$ and $H_J^{-}(M)$. $J$ is said to be $C^{\infty}$ pure if $H_J^{+}(M)\cap H_J^{-}(M)=\{0\}$, $C^{\infty}$ full if $H_J^{+}(M)+H_J^{-}(M)=H^2(M,\mathbb{R})$. Draghici, Li and Zhang \cite{DLZ1} concluded that for a 4 dimensional almost complex manifold $(M,J)$, $J$ is $C^{\infty}$ pure and full, i.e.:
$$H^2(M,\mathbb{R})=H^{+}_{J}(M)\oplus H^{-}_{J}(M).$$ Moreover, they consider the complexified cohomology group $H^2(M,\mathbb{C})=H^2(M,\mathbb{R})\otimes\mathbb{C}$, and get that if $J$ is integrable, $$H^2(M,\mathbb{C})=H^{1,1}_J\oplus H^{2,0}_J\oplus H^{0,2}_J.$$

For higher dimensional case, please refer to \cite{FT, TWZ} and references therein.

As we all know almost complex manifolds are always of even real dimension. For odd dimensional case, we can consider contact manifolds. In this paper, we consider the decomposition of degree 2 basic cohomology group $H^{2}_B(\mathcal{F}_{\xi})$ of a compact K-contact 5-manifold $(M,\xi,\eta,\Phi,g)$. There are two subgroups $H_{\Phi}^{+}$ and $H_{\Phi}^{-}$ of $H^{2}_B(\mathcal{F}_{\xi})$ which are called $\Phi$-invariant and $\Phi$-anti-invariant basic cohomology group respectively. $\Phi$ is defined to be $C^{\infty}$ pure if $H_{\Phi}^{+}\cap H_{\Phi}^{-}=\{0\}$, $C^{\infty}$ full if $H_{\Phi}^{+}+H_{\Phi}^{-}=H^2_B(\mathcal{F}_{\xi})$. We conclude that $\Phi$ is $C^{\infty}$ pure and full, i.e. Theorem \ref{main1}. Moreover, when $M$ is Sasakian, $\Phi$ is complex $C^{\infty}$ pure and full, i.e. Theorem \ref{main2}.

\section{5 dimensional K-contact manifolds}

Let us first recall some basic facts of K-contact and Sasakian manifolds. For details, please refer to \cite{BG,FHW}.

 Suppose $(M,\xi,\eta,\Phi,g)$ is a $2n+1$ dimensional compact K-contact manifold, here $\eta$ is the contact 1-form satisfying $\eta\wedge (d\eta)^n\not=0$ everywhere on $M$, $\xi$ is the Reeb vector field satisfying $\eta(\xi)=1$ and $\iota_{\xi}d\eta=0$, $\Phi\in End(TM)$ such that $\Phi\circ\Phi=-id+\xi\otimes\eta$. $(\xi,\eta,\Phi)$ is called an almost contact structure on $M$. $g$ is a Riemannian metric compatible with the almost contact structure $(\xi,\eta,\Phi)$ in the sense that $g(\Phi X,\Phi Y)=g(X,Y)-\eta(X)\eta(Y)$ and $g(X,\Phi Y)=d\eta(X,Y)$. The contact metric structure $\mathcal{S}=(\xi,\eta,\Phi,g)$ is called K-contact if $\xi$ is a Killing vector field of $g$, i.e., $L_{\xi}g=0$, where $L$ stands for the Lie derivative.

The Reeb vector field $\xi$ is often called the characteristic vector field and uniquely determines a 1-dimensional foliation $\mathcal{F}_{\xi}$ on $M$. The line bundle $L_{\xi}$ consists of tangent vectors that are tangent to the leaves of $\mathcal{F}_{\xi}$, and the contact subbundle $D$ is a codimension 1 subbundle of $TM$ whose fibers are the kernel of $\eta$. Then we have: $$TM=L_{\xi}\oplus D.$$

Consider the cone on $M$ as $C(M)=M\times \mathbb{R}^{+}$ with warped product metric $g_{C(M)}=\textrm{d}r^2+r^2g$. Let $\Upsilon=r\frac{\partial}{\partial r}$ be the Liouville vector field. For the almost contact structure $(\xi,\eta,\Phi)$ on $M$, an almost complex structure $J$ on $C(M)$ can be defined as a section of the endomorphism bundle of the tangent bundle $TC(M)$ satisfying:
$$JY=\Phi Y+\eta(Y)\Upsilon, J\Upsilon=-\xi.$$

$(\xi,\eta,\Phi)$ is said to be normal if the corresponding almost complex structure $J$ on $C(M)$ is integrable, and a normal contact metric structure $\mathcal{S}=(\xi,\eta,\Phi,g)$ on $M$ is called a Sasakian structure. Moreover, a pair $(M,\mathcal{S})$ is called a Sasakian manifold.

Suppose $(M,\xi,\eta,\Phi,g)$ is a K-contact manifold, a differential $p-$form $\alpha$ on $M$ is said to be basic if $$\iota_{\xi}\alpha=0, L_{\xi}\alpha=0.$$

We denote by $\Omega_B^p(\mathcal{F_{\xi}})$ the basic $p$-forms. It is easy to check that the exterior derivation $d$ takes basic forms to basic forms, so the subalgebra $\displaystyle{\Omega_B(\mathcal{F_{\xi}})=\oplus_{p}\Omega_B^p(\mathcal{F_{\xi}})}$ forms a subcomplex of the de Rham complex. Its cohomology ring $H^{\ast}_B(\mathcal{F}_{\xi})$ is defined to be the basic cohomology ring of $\mathcal{F}_{\xi}$. In the following we set $d_B=d|_{\Omega_B(\mathcal{F}_{\xi})}$. For any $\alpha\in \Omega_{B}^{p}(\mathcal{F_{\xi}})$, the transverse Hodge star operator $\bar{\ast}$ can be defined as follows:
$$\bar{\ast}\alpha=\ast(\eta\wedge\alpha)=(-1)^p \iota_{\xi}\ast\alpha.$$

The adjoint operator $\delta_B: \Omega_{B}^{p}(\mathcal{F_{\xi}})\rightarrow \Omega_{B}^{p-1}(\mathcal{F_{\xi}})$ of the basic differential operator $d_B$:
$$\delta_B=-\bar{\ast}d_B\bar{\ast}.$$

The basic Laplacian $\Delta_B$:
$$\Delta_B=d_B\delta_B+\delta_Bd_B.$$

Analogue to the Hodge decomposition of compact Riemannian manifolds, we have the transverse Hodge decomposition \cite{EKAH,KT,Ton}:
$$\Omega_{B}^{p}(\mathcal{F_{\xi}})=\mathcal{H}^p(\mathcal{F}_{\xi})\oplus \textrm{Im}(d_B)\oplus \textrm{Ker}(\delta_B),$$
where $\displaystyle{\mathcal{H}^p(\mathcal{F}_{\xi})}$ is the space of basic harmonic $p$-forms defined as the kernel of
$$\Delta_B: \Omega_{B}^{p}(\mathcal{F_{\xi}})\rightarrow \Omega_{B}^{p}(\mathcal{F_{\xi}}).$$
Specifically, for five dimensional K-contact manifold $(M,\xi,\eta,\Phi,g)$, one considers the contact subbundle $D$ with bundle metric $g_D$ induced by $g$. For simplicity, we still denote $g_D$ by $g$. The operators
$$\frac{1}{2}(\textrm{id}+\bar{\ast}), \frac{1}{2}(\textrm{id}-\bar{\ast})$$ induces a decomposition of the exterior bundle $\Lambda_D$ of $D$ by decompose any $\alpha$ into $\displaystyle{\frac{1}{2}(\alpha\pm\bar{\ast}\alpha)}$:
$$\Lambda^2_D=\Lambda^{+}_{g}\oplus\Lambda^{-}_{g}.$$
Denote by $\displaystyle{\Omega_{g}^{\pm}(\mathcal{F_{\xi}})}$ the relevant space of basic forms. Hence,
$$\Omega_{g}^{2}(\mathcal{F_{\xi}})=\Omega_{g}^{+}(\mathcal{F_{\xi}})\oplus\Omega_{g}^{-}(\mathcal{F_{\xi}}).$$ We call elements in $\Omega_{g}^{+}(\mathcal{F_{\xi}})$ and $\Omega_{g}^{-}(\mathcal{F_{\xi}})$ the basic self-dual and basic anti-self-dual forms. Moreover, $\Phi$ acts on the bundle of $\Lambda^2_{D}$ by $\alpha(\cdot,\cdot)\rightarrow\alpha(\Phi\cdot,\Phi\cdot)$, so we have the splitting by decomposition $\displaystyle{\alpha(\cdot,\cdot)=\frac{1}{2}[\alpha(\cdot,\cdot)\pm\alpha(\Phi\cdot,\Phi\cdot)]}$: $$\Lambda^2_D=\Lambda_{\Phi}^{+}\oplus\Lambda_{\Phi}^{-}.$$

We denote by $\Omega_{\Phi}^{+}(\mathcal{F}_{\xi})$ the space of $\Phi$-invariant basic 2-forms, $\Omega_{\Phi}^{-}(\mathcal{F}_{\xi})$ the space of $\Phi$-anti-invariant basic 2-forms. Then the $\Phi$-invariant and $\Phi$-anti-invariant basic cohomology groups can be defined as follows respectively:
\begin{eqnarray*}
&&H^{+}_{\Phi}(\mathcal{F}_{\xi})=\{ [\alpha]\in H^{2}_{\Phi}(\mathcal{F}_{\xi}) ~| ~\alpha\in\Omega_{\Phi}^{+}(\mathcal{F}_{\xi}) \};\\
&&H^{-}_{\Phi}(\mathcal{F}_{\xi})=\{ [\alpha]\in H^{2}_{\Phi}(\mathcal{F}_{\xi}) ~| ~\alpha\in\Omega_{\Phi}^{-}(\mathcal{F}_{\xi}) \}.
\end{eqnarray*}

For a basic form $\alpha$, we denote $\alpha_h$, $(\alpha)_{g}^\mp$ and $(\alpha)_{\Phi}^\pm$ the relevant basic harmonic, basic (anti-)self-dual and $\Phi$(-anti)-invariant part of $\alpha$ respectively.

With the notations of basic (anti-)self-dual forms, we have the following refined transverse Hodge decomposition:

\begin{lem}\label{hodge}
If $\alpha\in\Omega_{g}^{+}$ and $\alpha=\alpha_h+d_B\theta+\delta_B\Psi$, then $(d_B\theta)^{+}_{g}=(\delta_B\Psi)^{+}_{g}$ and $(d_B\theta)^{-}_{g}=-(\delta_B\Psi)^{-}_{g}$. In particular, $$\alpha-2(d_B\theta)^{+}_{g}=\alpha_h,$$ and $\displaystyle{\alpha+2(d_B\theta)^{-}_{g}=\alpha_h+2d_B\theta}$ is closed.
\end{lem}
\proof
By the basic Hodge decomposition: $\displaystyle{\alpha=\alpha_h+d_B\theta+\delta_B\Psi}$, there holds
$$\bar{\ast}\alpha=\bar{\ast}\alpha_h+\bar{\ast}d_B\theta+\bar{\ast}\delta_B\Psi.$$
Here $\bar{\ast}\alpha_h$ is harmonic, since $\Delta_B\bar{\ast}\alpha_h=\bar{\ast}\Delta_B\alpha_h=0$, and
$\bar{\ast}\delta_B\Psi=\bar{\ast}(\bar{\ast}d_B\bar{\ast})\Psi=d_B\bar{\ast}\Psi$. Hence, $\bar{\ast}\delta_B\Psi=d_B\theta$, and furthermore, $\bar{\ast}d_B\theta=\delta_B\Psi$. Then,
\begin{eqnarray*}
&&(d_B\theta)^{+}_{g}=\frac{1}{2}(\textrm{id}+\bar{\ast})(d_B\theta)=\frac{1}{2}(\textrm{id}+\bar{\ast})\bar{\ast}(\delta_B\Psi)
=(\delta_B\Psi)^{+}_{g};\\
&&(d_B\theta)^{-}_{g}=\frac{1}{2}(\textrm{id}-\bar{\ast})(d_B\theta)=\frac{1}{2}(\textrm{id}-\bar{\ast})\bar{\ast}(\delta_B\Psi)
=-(\delta_B\Psi)^{-}_{g}.
\end{eqnarray*}
Therefore,
\begin{eqnarray*}
\alpha&=&\alpha_h+(d_B\theta)^{+}_{g}+(d_B\theta)^{-}_{g}+(\delta_B\Psi)^{+}_{g}+(\delta_B\Psi)^{-}_{g}\\
&=&\alpha_h+2(d_B\theta)^{+}_{g}.
\end{eqnarray*}
Similarly, $\alpha+2(d_B\theta)^{-}_{g}=\alpha_h+2d_B\theta$.
\endproof

According to He \cite{He}, choose a coordinate $\{ x,y_1,y_2,y_3,y_4 \}$ such that, the frame $\{ e,e_1,e_2,\Phi e_1,\Phi e_2 \}$ is an adapted orthonormal frame. Its dual frame is $\{ \eta, \theta_1, \theta_2, \Phi\theta_1, \Phi\theta_2 \}$. Then $\omega=\frac{1}{2}d\eta=\theta_1\wedge\Phi\theta_1+\theta_2\wedge\Phi\theta_2$, and $\eta\wedge(\frac{1}{2}d\eta)^2=2\eta\wedge\theta_1\wedge\Phi\theta_1\wedge\theta_2\wedge\Phi\theta_2$ is twice volume form.

Since $\Phi\omega=\omega, \bar{\ast}\omega=\omega$, and we have the following equalities:
\begin{eqnarray*}
\Lambda^{+}_{\Phi}&=&span\{ \theta_1\wedge\Phi\theta_1, \theta_2\wedge\Phi\theta_2, \theta_1\wedge\theta_2+\Phi\theta_1\wedge\Phi\theta_2, \theta_1\wedge\Phi\theta_2-\Phi\theta_1\wedge\theta_2 \};\\
\Lambda^{-}_{\Phi}&=&span\{ \theta_1\wedge\theta_2-\Phi\theta_1\wedge\Phi\theta_2, \theta_1\wedge\Phi\theta_2+\Phi\theta_1\wedge\theta_2 \};\\
\Lambda^{+}_{g_D}&=&span\{ \theta_1\wedge\Phi\theta_1-\theta_2\wedge\Phi\theta_2, \theta_1\wedge\theta_2-\Phi\theta_1\wedge\Phi\theta_2, \theta_1\wedge\Phi\theta_2+\Phi\theta_1\wedge\theta_2 \};\\
\Lambda^{-}_{g_D}&=&span\{ \theta_1\wedge\Phi\theta_1-\theta_2\wedge\Phi\theta_2, \theta_1\wedge\theta_2+\Phi\theta_1\wedge\Phi\theta_2, \theta_1\wedge\Phi\theta_2-\Phi\theta_1\wedge\theta_2 \},
\end{eqnarray*}

there hold the following equalities:
\begin{eqnarray*}
&&\Lambda^{+}_{\Phi}=\mathbb{R}\omega\oplus\Lambda^{-}_{g}, \Lambda^{+}_{g}=\mathbb{R}\omega\oplus\Lambda^{-}_{\Phi};\\
&&\Lambda^{+}_{\Phi}\cap\Lambda^{+}_{g}=\mathbb{R}\omega, \Lambda^{-}_{\Phi}\cap\Lambda^{-}_{g}=0.
\end{eqnarray*}

We denote by $\mathcal{Z}^{-}_{\Phi}$ the set of closed $\Phi-$anti-invariant 2-forms, $\mathcal{H}^{+}_{g}$ the set of basic harmonic self-dual 2-forms, and $\mathcal{H}^{+,\omega^{\perp}}_{g}$ the set of basic harmonic self-dual 2-forms which are perpendicular to $\omega$ with respect to the metric induced by $g$. Then we have:
\begin{lem}
$\mathcal{Z}^{-}_{\Phi}\subset\mathcal{H}^{+}_{g}$, and $\mathcal{Z}^{-}_{\Phi}\subset H^{-}_{\Phi}$ is bijective. Furthermore, $H^{-}_{\Phi}=\mathcal{Z}^{-}_{\Phi}=\mathcal{H}^{+,\omega^{\perp}}_{g}$.
\end{lem}
\proof
If $\alpha\in\mathcal{Z}^{-}_{\Phi}$, then $d\alpha=0$. Since $\alpha$ is self dual, i.e., $\bar{\ast}\alpha=\alpha$, $\delta_B\alpha=\bar{\ast}d_B\bar{\ast}\alpha=\bar{\ast}d_B\alpha=0$, i.e., $\alpha\in\mathcal{H}^{+}_{g}$.

By $\Lambda^{+}_{g}=\mathbb{R}\omega\oplus\Lambda^{-}_{\Phi}$ and $\mathcal{Z}^{-}_{\Phi}= H^{-}_{\Phi}$, we have $H^{-}_{\Phi}=\mathcal{Z}^{-}_{\Phi}=\mathcal{H}^{+,\omega^{\perp}}_{g}$.
\endproof

Based on the above lemmas, we conclude the following theorem:
\begin{thm}\label{main1}
For a five dimensional closed K-contact manifold $(M,\xi,\eta,\Phi,g)$, $\Phi$ is $C^{\infty}$ pure and full in the following sense:
$$H^2_{B}(\mathcal{F}_{\xi})=H^{+}_{\Phi}\oplus H^{-}_{\Phi}.$$
\end{thm}
\proof
If $\mathfrak{a}\in H^{+}_{\Phi}\cap H^{-}_{\Phi}$, let $\alpha^{'}\in\mathcal{Z}^{+}_{\Phi}$, $\alpha^{''}\in\mathcal{Z}^{-}_{\Phi}$ be representatives for $\mathfrak{a}$, $\alpha^{'}=\alpha^{''}+d_B\gamma$ for some basic 1-form $\gamma$. Then
\begin{eqnarray*}
0&=&\int_M \alpha^{'}\wedge\alpha^{''}\wedge\eta\\
&=&\int_M (\alpha^{''}+d_B\gamma)\wedge\alpha^{''}\wedge\eta\\
&=&\int_M \alpha^{''}\wedge\alpha^{''}\wedge\eta+\int_M d_B\gamma\wedge\alpha^{''}\wedge\eta\\
&=&\int_M \alpha^{''}\wedge \alpha^{''}\wedge\eta+\int_M \gamma\wedge d_B\alpha^{''}\wedge\eta-\int_M \gamma\wedge\alpha^{''}\wedge d_B\eta\\
&=&\int_M \alpha^{''}\wedge \bar{\ast}\alpha^{''}\wedge\eta\\
&=&\int_M |\alpha^{''}|^2_{g}~\textrm{d}vol,
\end{eqnarray*}
since $\gamma\wedge\alpha^{''}\wedge d_B\eta$ is a basic 5-form, it is zero, and $\alpha^{''}$ is basic self-dual form, satisfies $\bar{\ast}\alpha^{''}=\alpha^{''}$.

Hence, $\alpha^{''}=0$, i.e., $\mathfrak{a}=0$.

Next, we prove $H^2(\mathcal{F}_{\xi})=H^{+}_{\Phi}\oplus H^{-}_{\Phi}$. Suppose the contrary, then there exists $\mathfrak{b}\in H^2(\mathcal{F}_{\xi})\setminus H^{+}_{\Phi}\oplus H^{-}_{\Phi}$. Since $H^{-}_g\subset H^{+}_{\Phi}$, assume $\mathfrak{b}\in H^{+}_g$. Let $\beta$ be the basic harmonic, self-dual representative of $\mathfrak{b}$, and denote $f=\langle\beta, \omega\rangle$. Then $f\not =0$. Otherwise, $\mathfrak{b}\in H^{-}_{\Phi}$. Consider the basic self-dual form $f\omega$. By Lemma \ref{hodge}, $(f\omega)_h+2(f\omega)^{exact}=f\omega+2[(f\omega)^{exact}]^{-}_g$ is closed and $\Phi$-invariant. Thus, $\mathfrak{c}=[(f\omega)_h+2(f\omega)^{exact}]\in H^{+}_{\Phi}$.  However,
\begin{eqnarray*}
&&\int \beta\wedge[(f\omega)_h+2(f\omega)^{exact}]\wedge \eta\\
&=&\int \langle\beta,(f\omega)_h+2(f\omega)^{exact}\rangle d\mu\\
&=&\int \langle\beta,f\omega+2((f\omega)^{exact})^{-}_{g}\rangle d\mu\\
&=&\int f^2 d\mu\\
&\not =&0.
\end{eqnarray*}

This contradicts the assumption that $\mathfrak{b}\bot H^{+}_{\Phi}\oplus H^{-}_{\Phi}$.

\endproof

\section{5 dimensional Sasakian manifolds}

We consider the complex basic 2-forms in this section. There holds the following decomposition:
$$\Lambda^2_{D, \mathbb{{C}}}=\Lambda^{2,0}_{\Phi}\oplus\Lambda^{1,1}_{\Phi}\oplus\Lambda^{0,2}_{\Phi}.$$

Let $\omega^i=\theta^{i}+\sqrt{-1}\Phi\theta^{i}$. Then:
\begin{eqnarray*}
(\Lambda^{1,1}_{\Phi})_{\mathbb{R}}&=&span\{ \sqrt{-1}\omega^1\wedge\overline{\omega}^1, \sqrt{-1}\omega^2\wedge\overline{\omega}^2,
\omega^1\wedge\overline{\omega}^2+\overline{\omega}^1\wedge\omega^2, \\
&&\sqrt{-1}(\omega^1\wedge\overline{\omega}^2-\overline{\omega}^1\wedge\omega^2) \},\\
(\Lambda^{2,0}_{\Phi}\oplus\Lambda^{0,2}_{\Phi})_{\mathbb{R}}&=&span\{ \omega^1\wedge\omega^2+\overline{\omega}^1\wedge\overline{\omega}^2,
\sqrt{-1}(\omega^1\wedge\omega^2-\overline{\omega}^1\wedge\overline{\omega}^2) \}.
\end{eqnarray*}
By a direct calculation we have:
\begin{eqnarray}
&&\Lambda^{+}_{\Phi}=(\Lambda^{1,1}_{\Phi})_{\mathbb{R}},\label{6}\\ &&\Lambda^{-}_{\Phi}=(\Lambda^{2,0}_{\Phi}\oplus\Lambda^{0,2}_{\Phi})_{\mathbb{R}}.\label{6}
\end{eqnarray}

\begin{defn}
Let $H^{p,q}_{\Phi}$ be the subspace of the complexified basic cohomology $H^2_{B}(\mathcal{F}_{\xi};\mathbb{C})=H^2_{B}(\mathcal{F}_{\xi};\mathbb{R})\otimes\mathbb{C}$, consisting of classes which can be represented by a complex closed form of type $(p,q)$.
\end{defn}

\begin{lem}
There hold the following properties of the subgroups $H^{p,q}_{\Phi}$:
\begin{eqnarray}
&&H^{p,q}_{\Phi}=\overline{H^{q,p}_{\Phi}};\label{1}\\
&&H^{p,p}_{\Phi}=(H^{p,p}_{\Phi}\cap H^{2p}_{B}(\mathcal{F}_{\xi};\mathbb{R}))\otimes\mathbb{C};\label{2}\\
&&(H^{p,q}_{\Phi}+H^{q,p}_{\Phi})=((H^{p,q}_{\Phi}+H^{q,p}_{\Phi})\cap H^{p+q}_{B}(\mathcal{F}_{\xi};\mathbb{R}))\otimes\mathbb{C}.\label{3}
\end{eqnarray}
\end{lem}
\begin{proof}
Choose a complex form $\Psi$, then (\ref{1}) follows from the fact that $\Psi$ is closed if and only if its conjugate $\overline{\Psi}$ is closed. (\ref{2}) and (\ref{3}) follow from a fact in linear algebra:

Let $V$ be a real vector space and $W$ a complex subspace of $V\otimes_{\mathbb{R}}\mathbb{C}$, and $W$ is invariant under conjugation as a subspace. Then $W=(W\cap V)\otimes \mathbb{C}.$ See \cite{BHPV}.
\end{proof}

\begin{lem}\label{a}
For a compact 5-dimensional K-contact manifold, there hold the following:
\begin{eqnarray}
&&H^{+}_{\Phi}=H^{1,1}_{\Phi}\cap H^2(\mathcal{F}_{\xi};\mathbb{R});\label{4}\\
&&H^{1,1}_{\Phi}=H^{+}_{\Phi}\otimes_{\mathbb{R}}\mathbb{C}.\label{5}
\end{eqnarray}
\end{lem}
\begin{proof}
We first prove (\ref{4}). By (\ref{6}) we have $H^{+}_{\Phi}\subseteq H^{1,1}_{\Phi}\cap H^2(\mathcal{F}_{\xi};\mathbb{R})$. For the converse inclusion, we choose an element $[\rho]\in H^{1,1}_{\Phi}\cap H^2(\mathcal{F}_{\xi};\mathbb{R})$, such that $\rho$ is a $d_B$ closed basic $(1,1)$ form of the form $\rho=\sigma+d_B\tau$, where $\sigma$ a $d_B$ closed basic real form. Hence, $[\rho]$ is also represented by the real $d_B$ closed basic $(1,1)$ form $\displaystyle{\frac{1}{2}(\rho+\overline{\rho})=\sigma+d_B(\frac{\tau+\overline{\tau}}{2})}$. This shows that $H^{+}_{\Phi}\supseteq H^{1,1}_{\Phi}\cap H^2(\mathcal{F}_{\xi};\mathbb{R})$.

The relation (\ref{5}) is a direct consequence of (\ref{2}) with $p=1$ and (\ref{4}).
\end{proof}

\begin{lem}\label{b}
 Suppose that $M$ is a compact 5-dimensional K-contact manifold. If the contact metric structure $\mathcal{S}=(\xi,\eta,\Phi,g)$ is normal, i.e., $(M,\mathcal{S})$ is Sasakian, there hold the following:
\begin{eqnarray}
&&(H^{2,0}_{\Phi}+H^{0,2}_{\Phi})=H^{-}_{\Phi}\otimes_{\mathbb{R}}\mathbb{C};\label{7}\\
&&H^{-}_{\Phi}=(H^{2,0}_{\Phi}+H^{0,2}_{\Phi})\cap H^2(\mathcal{F}_{\xi};\mathbb{R}).\label{8}
\end{eqnarray}
\end{lem}
\begin{proof}
Choose a complex form $\Theta=\alpha+i\Phi\alpha\in\Omega^{2,0}_{\Phi}$, where $\alpha\in\Omega^{-}_{\Phi}$. Since $d_B=\partial_B+\overline{\partial}_B$ and $2\alpha=\Theta+\overline{\Theta}$, we have $$2d_B\alpha=(\partial_B+\overline{\partial}_B)(\Theta+\overline{\Theta})=\partial_B\overline{\Theta}+\overline{\partial}_B\Theta.$$ Here we have used the fact that $\partial_B\Theta=0=\overline{\partial}_B\overline{\Theta}$, since $M$ is 5-dimensional and $\partial_B\Theta$ is a basic $(3,0)$ form, $\overline{\partial}_B\overline{\Theta}$ is a basic $(0,3)$ form. Therefore, $$d_B\alpha=0\Leftrightarrow\partial_B\overline{\Theta}=0\Leftrightarrow\overline{\partial}_B\Theta=0.$$
Similarly, $$d_B(\Phi\alpha)=0\Leftrightarrow\partial_B(\overline{i\Theta})=0\Leftrightarrow\overline{\partial}_B(i\Theta)=0.$$
Moreover, $\overline{\partial}_B\Theta=0$ if and only if $\overline{\partial}_B(i\Theta)=0$. Then it follows that $d_B\alpha=0$ if and only if $ d_B(\Phi\alpha)=0$. Therefore, $(H^{2,0}_{\Phi}+H^{0,2}_{\Phi})=H^{-}_{\Phi}\otimes_{\mathbb{R}}\mathbb{C}$.

The relation (\ref{8}) follows from (\ref{3}) with $(p,q)=(2,0)$ and (\ref{7}).
\end{proof}

Next we suppose the contact metric structure $\mathcal{S}=(\xi,\eta,\Phi,g)$. Combining with Lemma \ref{a} and Lemma \ref{b}, there holds the following:

\begin{thm}\label{main2}
For a compact 5-dimensional K-contact manifold $(M,\mathcal{S})$, $\Phi$ is always complex $C^{\infty}$ pure in the sense:
$$H^{1,1}_{\Phi}\cap H^{2,0}_{\Phi}\cap H^{0,2}_{\Phi}=\{0\}.$$
Moreover, if $\mathcal{S}$ is normal, then $\Phi$ is also complex $C^{\infty}$ full in the sense:
$$H^2(\mathcal{F}_{\xi};\mathbb{C})=H^{1,1}_{\Phi}\oplus H^{2,0}_{\Phi}\oplus H^{0,2}_{\Phi}$$
\end{thm}

\vspace{0.5cm}\noindent\textbf{Acknowledgments}.  The authors would like to thank professor Hongyu Wang for useful discussion. The first author was partially supported by NSFC Grant 11426195. The second author was partially supported by NSFC Grant 11471145 and Qing Lan Project.

\end {document}